\documentclass[12pt]{amsart}
\usepackage{amssymb}
\usepackage[dvips]{color}
\usepackage{graphicx}
\usepackage{epsfig}
\usepackage[all]{xy}
\usepackage{color}
\newtheorem{theorem}{Theorem}[section]
\newtheorem{lemma}[theorem]{Lemma}

\newtheorem{corollary}[theorem]{Corollary}

\theoremstyle{definition}
\newtheorem{example}[theorem]{Example}

\theoremstyle{remark}

\newtheorem{remark}[theorem]{Remark}

\def\X{\,\,\lower2pt\hbox{\input{figX.pstex_t}}}
\def\noXv{\,\,\lower2pt\hbox{
\begin{picture}(0,0)%
\includegraphics{figNoX.pstex}%
\end{picture}%
\setlength{\unitlength}{1973sp}%
\begingroup\makeatletter\ifx\SetFigFont\undefined%
\gdef\SetFigFont#1#2#3#4#5{%
  \reset@font\fontsize{#1}{#2pt}%
  \fontfamily{#3}\fontseries{#4}\fontshape{#5}%
  \selectfont}%
\fi\endgroup%
\begin{picture}(316,316)(293,-969)
\end{picture}
}}
\def\noXh{\,\,\lower2pt\hbox{\input{figNoX2.pstex_t}}}
\def\noXDU{\,\,\lower2pt\hbox{\input{figDU.pstex_t}}}
\def\noXDD{\,\,\lower2pt\hbox{\input{figDD.pstex_t}}}
\def\noXUD{\,\,\lower2pt\hbox{\input{figUD.pstex_t}}}
\def\noXUU{\,\,\lower2pt\hbox{\input{figUU.pstex_t}}}
\def\noXRR{\,\,\lower2pt\hbox{\input{figRR.pstex_t}}}
\def\noXRL{\,\,\lower2pt\hbox{\input{figRL.pstex_t}}}
\def\noXLR{\,\,\lower2pt\hbox{\input{figLR.pstex_t}}}
\def\noXLL{\,\,\lower2pt\hbox{\input{figLL.pstex_t}}}
\def\XRR{\,\,\lower2pt\hbox{\input{figXRR.pstex_t}}}
\def\XRL{\,\,\lower2pt\hbox{\input{figXRL.pstex_t}}}
\def\XLR{\,\,\lower2pt\hbox{\input{figXLR.pstex_t}}}
\def\XLL{\,\,\lower2pt\hbox{\input{figXLL.pstex_t}}}

\newlength{\cellsz}
\newcounter{cellsize}
\newcommand{\setcellsize}[1]{%
  \setcounter{cellsize}{#1}%
  \setlength{\cellsz}{\value{cellsize}\unitlength}}%
\setcellsize{13}%
\newcommand\cellify[1]{\def\thearg{#1}\def\nothing{}%
\ifx\thearg\nothing \vrule width0pt height\cellsz depth0pt\else
\hbox to 0pt{{\begin{picture}(\value{cellsize},\value{cellsize})
  \put(0,0){\line(1,0){\value{cellsize}}}
  \put(0,0){\line(0,1){\value{cellsize}}}
  \put(\value{cellsize},0){\line(0,1){\value{cellsize}}}
  \put(0,\value{cellsize}){\line(1,0){\value{cellsize}}} \end{picture} \hss}}\fi%
\vbox to \cellsz{ \vss \hbox to \cellsz{\hss$#1$\hss} \vss}}
\newcommand\tableau[1]{\vcenter{\vbox{\let\\\cr
\baselineskip -16000pt \lineskiplimit 16000pt \lineskip 0pt
\ialign{&\cellify{##}\cr#1\crcr}}}}
\newcommand\tabl[1]{\vtop{\let\\\cr
\baselineskip -16000pt \lineskiplimit 16000pt \lineskip 0pt
\ialign{&\cellify{##}\cr#1\crcr}}}

\def\Imm{\mathrm{Imm}}

\begin{document}
 
\title{$A_2$-web immanants}
\author{Pavlo Pylyavskyy}

\address{Department of Mathematics, University of Michigan, Ann Arbor, MI, 48103}
\email{pavlo@umich.edu}

\begin{abstract}
We describe the rank $3$ Temperley-Lieb-Martin algebras in terms of Kuperberg's $A_2$-webs. We define consistent labelings of webs, and use them to describe the coefficients of decompositions into irreducible webs. We introduce web immanants, inspired by Temperley-Lieb immanants of Rhoades and Skandera. We show that web immanants are positive when evaluated on totally positive matrices, and describe some further properties.
\end{abstract}

\maketitle

\section{Introduction}

Temperley-Lieb algebras are quotients of Hecke algebras such that only the irreducible representations corresponding to Young shapes with at most two columns survive. Originally introduced in \cite{TL} for the study of percolation, Temperley-Lieb algebras appeared in many other contexts. In particular, Rhoades and Skandera in \cite{RS1} used them to introduce Temperley-Lieb immanants, which are functions on matrices possessing certain positivity properties. In \cite{LPP} those immanants and their further properties developed in \cite{RS2}, were used to resolve some Schur-positivity conjectures. In \cite{LP} Temperley-Lieb pfaffinants were introduced, which can be viewed as super- analogs of Temperley-Lieb immanants. 

In this work we generalize in a different direction. Namely, we make use of multi-column generalizations of Temperley-Lieb algebras. Those are the Temperley-Lieb-Martin algebras (or TLM algebras) introduced by Martin in \cite{M}. Their irreducible representations correspond to Young shapes with at most $k$ columns.

In \cite{BK} Brzezi\'nski and Katriel gave a description of TLM algebras in terms of generators and relations. However, in order to deal with the combinatorics of TLM algebras one desires more than that: it is natural to ask whether a diagrammatic calculus exists for TLM algebras similar to that of Kauffman diagrams for Temperley-Lieb algebras. It appears that the $A_2$ spiders of Kuperberg \cite{K} essentially provide such calculus for $k=3$.

The paper proceeds as follows. In Section \ref{sec:web} we review the presentation of TLM algebras obtained in \cite{BK}. We proceed to define a diagrammatic calculus for TLM algebras using the spider reduction rules of \cite{K}. This allows us to introduce the web bases of TLM algebras. In Section \ref{sec:cons} we introduce the tool of consistent labelings of webs, which allows us to describe the coefficients involved in the decomposition of reducible webs into irreducible ones. In Section \ref{sec:imm} we introduce web immanants. We show that web immanants are positive when evaluated on totally positive networks. In Section \ref{sec:min} we give a positive combinatorial formula for decomposing products of triples of complementary minors into web immanants. In Section \ref{sec:tl} we relate web immanants and Temperley-Lieb immanants. In Section \ref{sec:net} we use the setting of weighted planar networks to give an interpretation of web immanants, thus providing a generalization of the Lindstr\"om lemma. In Section \ref{sec:con} we discuss potential further directions. 

The author would like to express gratitude to the following people: Thomas Lam for encouragement and helpful comments, Bruce Westbury for insightful comments on a draft, Richard Stanley for pointing out the bijection in the proof of Theorem \ref{thm:eta}, Mark Skandera and Greg Kuperberg for feedback on a draft, T. Kyle Petersen for his generous help with proofreading. 

\section{Web bases of TLM algebras} \label{sec:web}

The {\it {Hecke algebra}} $H_n(q)$ is a free associative algebra over $\mathbb C[q]$ generated by elements $g_1, \ldots, g_{n-1}$ subject to the following relations: $$g_i^2 = (q-1) g_i +q;$$ $$g_i g_{i+1} g_i = g_{i+1} g_i g_{i+1};$$ and $$g_i g_j = g_j g_i$$ if $|i-j|>1$. For a permutation $w \in S_n$ and a reduced decomposition $w = \prod s_{i_j}$ the element $g_w = \prod g_{i_j}$ does not depend on the choice of reduced decomposition. As $w$ runs over all permutations in $S_n$, elements $g_w$ form a linear basis for $H_n(q)$. Note that for $q=1$ the Hecke algebra is the group algebra $\mathbb C S_n$ of the symmetric group.

The {\it {Temperley-Lieb algebra}} $TL_n(q^{1/2}+q^{-1/2})$ is a $\mathbb C[q^{1/2}+q^{-1/2}]$-algebra generated by $e_1, \ldots, e_n$ with relations $$e_i^2=(q^{1/2} + q^{-1/2}) e_i,$$ $$e_ie_{i+1}e_i = e_{i} e_{i-1} e_{i} = e_i,$$ and $$e_i e_j = e_j e_i$$ for $|i-j|>1$. Temperley-Lieb algebras are quotients of the Hecke algebras in which only the irreducible modules corresponding to shapes with at most two columns survive. The map $\theta_2: g_i \mapsto q^{1/2}e_i - 1$ gives an algebra homomorphism. 

{\it {Temperley-Lieb-Martin algebras}} are quotients of the Hecke algebra such that only the representations with at most $k$ columns survive. Thus, for $k=2$ those are exactly the Temperley-Lieb algebras. In \cite{BK} the following presentation for a Temperley-Lieb-Martin algebra $TLM_n^k$ was given. Denote $$[k]_q = \frac{q^{k/2}-q^{-k/2}}{q^{1/2}-q^{-1/2}}.$$ For $i = 1, \ldots, k$ and $j = 1, \ldots, n-i$ we have generator $e_j^{(i)}$ subject to 
$$(e_j^{(i)})^2 = [i+1]_q e_j^{(i)},$$ 
$$e_j^{(i+1)} = \frac{1}{[i]_q [i+1]_q} (e_j^{(i)} e_{j+1}^{(i)} e_j^{(i)} - e_j^{(i)}) = \frac{1}{[i]_q [i+1]_q} (e_{j+1}^{(i)} e_{j}^{(i)} e_{j+1}^{(i)} - e_{j+1}^{(i)}),$$ 
and $$e_j^{(k)}=0.$$
An algebra homomorphism $\theta_k: H_n(q) \longrightarrow TLM_n^k(q^{1/2}+q^{-1/2})$ is given by $\theta_k(g_i) = q^{1/2}e_i^{(1)} - 1$.

The generators of $TL_n$ can be represented by {\it {Kauffman diagrams}}. Each diagram is a matching on $2n$ vertices arranged on opposite sides of a rectangle: $n$ on the left and $n$ on the right. Each $e_i$ is represented by a single uncrossing \noXv between the $i$-th and $i+1$-st elements. The product is given by concatenation, with loops being erased while contributing a factor of $q^{1/2} + q^{-1/2}$. It is known that if $w$ is a $(3,2,1)$-avoiding permutation and $w = \prod s_{i_k}$ is a reduced decomposition, then $e_w = \prod e_{i_k}$ does not depend on the choice of reduced decomposition. As $w$ runs over the set of $(3,2,1)$-avoiding permutations the $e_w$ form a basis for $TL_n$.

A natural question is whether there exists similar planar diagram presentation for TLM algebras. Such presentation for $k=3$ is implicit in \cite{K}. Namely, Kuperberg considered planar diagrams, or $A_2$-{\it {webs}}, with each inner vertex of degree $3$ and each boundary vertex of degree $1$. In addition, an orientation on edges of the web is given that makes every vertex either a source or a sink. The following {\it {spider reduction rules}} were introduced in \cite{K}: 

\begin{figure}[h!]
\begin{center}
\input{sp6.pstex_t}
\end{center}
\caption{}\label{fig:sp6}
\end{figure}

Unless specified otherwise the word web will refer to $A_2$-web in what follows.  It is known that every non-reduced web can be uniquely reduced to a linear combination of reduced webs using the above rules, see for example \cite[Corollary 5.1]{SW}.

\begin{figure}[h!]
\begin{center}
\input{sp1.pstex_t}
\end{center}
\caption{}\label{fig:sp1}
\end{figure}

\begin{remark}
Note that the webs introduced in \cite{K} have arbitrary boundary conditions, while we restrict our attention to the webs having $n$ sources on the left and $n$ sinks on the right. Note also that the rules in \cite{K} actually differ by sign. For a reason to be evident later we prefer the positive version. 
\end{remark}

Let $W_n$ be an algebra generated by the diagrams on Figure \ref{fig:sp1} with spider reduction rules. Let $\eta: TLM_n^3 \longrightarrow W_n$ be given by mapping $e_i^{(1)}$ into the first type of diagrams, and the elements $[2]_q e_i^{(2)}$ into the second kind of diagrams (note the coefficient).

\begin{theorem} \label{thm:eta}
Map $\eta$ is an algebra isomorphism.
\end{theorem}

\begin{proof}

The defining relations of $TLM_n^3$ are easily verified inside $W_n$, as seen on Figure \ref{fig:sp2}. 

\begin{figure}[h!]
\begin{center}
\input{sp2.pstex_t}
\end{center}
\caption{}\label{fig:sp2}
\end{figure}

On the other hand, according to \cite[Theorem 6.1]{K} the dimension of $W_n$ is equal to the number of $\mathfrak {sl}_3$-invariants $Inv(V_{(1)}^{\otimes n} \otimes V_{(1,1)}^{\otimes n})$, where $V_{(1)}$ and $V_{(1,1)}$ are the two fundamental representations of $\mathfrak {sl}_3$. This number is easily seen to be equal to the Kostka number $K_{3^n, 1^n2^n}$. The letter enumerates the semi-standard tableaux which are in bijection with pairs of standard Young tableaux of the same shape $\lambda$ such that $|\lambda|=n$ and $\lambda$ has at most $3$ columns. This number is exactly the dimension of $TLM_n^3$. It follows from RSK correspondence and the theory of Greene-Kleitman invariants that this number is equal to the number of $(4,3,2,1)$-avoiding permutations, cf. \cite{St}. 

Thus the dimensions of $W_n$ and $TLM_n^3$ are equal. We postpone the proof of injectivity till Theorem \ref{thm:rhom}. The two facts together imply that $\eta$ is an isomorphism.
\end{proof}

In particular, one can define elements $e_D = \eta^{-1}(D)$ of $TLM_n^3$ for each web $D$ occurring in $W_n$. As $D$ runs over irreducible webs $e_D$ form a {\it {web basis}} of $TLM_n^3$. Note that unlike in the case of Temperley-Lieb algebra, the $e_D$ are not always monomials in the $e_i^{(j)}$. For example in $TLM_4^3$ one of the webs can be expressed as $[2]_q (e_2^{(1)}e_1^{(1)}e_2^{(2)}-e_2^{(2)})$.

\section{Consistent labelings} \label{sec:cons}

\subsection{Definition and statistic}

Denote $\mathfrak M_n$ the set of webs $D$ of size $n$. Each edge $e \in D$ has two sides, which we denote $e_+$ and $e_-$, so that every edge is directed from its positive to its negative side.  A {\it {consistent labeling}} of $D$ is an assignment of a label $f: e_\pm \mapsto 1,2,3,1',2',3'$ to each side of each edge so that the following conditions hold:

\begin{enumerate}
    \item positive sides are labeled with $1,2,3$, negative sides are labeled with $1',2',3'$;
    \item if $e_+$ is labeled with $i$ then $e_-$ is labeled with $i'$;
    \item the labels adjacent to the same vertex are distinct, i.e. the sides of edges adjacent to every degree $3$ vertex in $D$ are labeled either with $1,2,3$ or with $1',2',3'$.
\end{enumerate}

The labels adjacent to boundary vertices of $D$ are called {\it {boundary labels}}. The restriction $g$ of a labeling $f$ to the boundary is called a {\it {boundary labeling}}. Let $L_D$ denote the set of all consistent labelings of $D$, and $L_{D,g}$ denote the set of all consistent labelings with a prescribed boundary labeling $g$.

\begin{figure}[h!]
\begin{center}
\input{sp16.pstex_t}
\end{center}
\caption{}\label{fig:sp16}
\end{figure}

An example of a consistent labeling is given in Figure \ref{fig:sp16}. For this web and this boundary labeling there exists only one consistent labeling, i.e., $|L_{D,g}|=1$.

Let a {\it {singularity}} of a consistent labeling be one of the following:
\begin{enumerate}
     \item a degree $3$ vertex in $D$;
     \item a point where an edge of $D$ is tangent to a vertical line.
\end{enumerate}

Readjusting the embedding of $D$ one can clearly make its edges non-vertical, and make no two singularities lie on one vertical line. 

Let $v$ be a singularity of the first kind. Let $p, q, r$ be labels adjacent to $v$, so they are either $1,2,3$ or $1',2',3'$. Define order on the labels as follows: $1<2<3$ and $3'<2'<1'$. Let $l_v$ be the vertical line passing through $v$. For an unordered pair of labels $(p,q)$ adjacent to $v$ define 

$$
\alpha(p,q) = 
\begin{cases}
-1 & \text{if $p<q$, $p$ and $q$ both lie to the left of $l_v$ and $p$ is above $q$;}\\
-1 & \text{if $p<q$, $p$ and $q$ both lie to the right of $l_v$ and $p$ is below $q$;}\\
1 & \text{if $p>q$, $p$ and $q$ both lie to the left of $l_v$ and $p$ is above $q$;}\\
1 & \text{if $p>q$, $p$ and $q$ both lie to the right of $l_v$ and $p$ is below $q$;}\\
0 & \text{if $p$ and $q$ lie on different sides of $l_v$.}
\end{cases}
$$

Let $\alpha(v) = \alpha(p,q) + \alpha(p,r) + \alpha(q,r)$ be the sum taken over all pairs of labels adjacent to $v$. 

Let $v$ be now a singularity of the second kind, and again let $l_v$ be the vertical line passing through $v$. Recall that each edge of $D$ is oriented from some $i$ to $i'$. Assume that at $v$ line $l_v$ is tangent to the edge labeled $i$ at the beginning, $i'$ at the end. Let 

$$
\alpha(v) = 
\begin{cases}
4-2i & \text{if edge is oriented down around $v$ and touches $l_v$ from the left;}\\
2i-4 & \text{if edge is oriented down around $v$ and touches $l_v$ from the right;}\\
4-2i & \text{if edge is oriented up around $v$ and touches $l_v$ from the right;}\\
2i-4 & \text{if edge is oriented up around $v$ and touches $l_v$ from the left;}
\end{cases}
$$

Now for a consistent labeling $f$ of $D$ define $$\alpha(f) = \prod_v q^{\frac{\alpha(v)}{4}},$$ where the product is taken over all singularities of a particular embedding of $D$. 

\begin{example}
The leftmost singularity $v$ on Figure \ref{fig:sp16} has $2'$ and $3'$ to the left of $l_v$, $2'$ above $3'$, and $1'$ to the right of $l_v$. Then $\alpha(2',3') = 1$ while $\alpha(1',2')=\alpha(1',3')=0$, which results in $\alpha(v) = 1$. For this embedding of the web  there are no singularities of the second kind and for this particular labeling $f$ we have $\alpha(f) = q^{-1/2}$.
\end{example}

The following is the key property of statistic $\alpha$.

\begin{lemma}
$\alpha(f)$ does not depend on the particular embedding of web $D$.
\end{lemma}

\begin{proof}
The first part of Figure \ref{fig:sp14} demonstrates that we can bend or unbend edges as long as we do not change the direction at the ends. Indeed, the two singularities created by such bending cancel out, contributing total of $q^{\frac{2i-4}{4}}q^{\frac{4-2i}{4}}=1$ into $\alpha(f)$. Other cases are similar.

Now if we do want to change the direction of one of the edges at its end, we get the situation of the type shown in second part of Figure \ref{fig:sp14}. There an edge changes its side with respect to $l_v$, and because of that a new singularity is created. For each pair of labels $j<i$ adjacent to $v$ the value $\alpha(i,j)$ becomes one less than it used to be: it either used to be $1$ and became $0$, or it used to be $0$ and became $-1$. Similarly for each $j>i$ the value of $\alpha(i,j)$ is one more than it used to be. That results in the total factor of $q^{-\frac{i-1}{4}}q^{\frac{3-i}{4}}$. This however cancels out with the new factor $q^{\frac{2i-4}{4}}$ coming from the new singularity. Other cases are similar.

\begin{figure}[h!]
\begin{center}
\input{sp14.pstex_t}
\end{center}
\caption{}\label{fig:sp14}
\end{figure}

It is easy to see now that any two embeddings of the same web can be deformed one into another by a sequence of moves of the above two types. This implies the statement.
\end{proof}

Denote $$|L_{D,g}|_q = \sum_{f \in L_{D,g}} \alpha(f),$$ we refer to $|L|_q$ as to $q$-{\it {size}} of $L$. Note that when $q=1$ the $q$-size $|L_{D,g}|_q = |L_{D,g}|$ is just the number of elements in $L_{D,g}$.

\subsection{Properties}

Let $G_n$ be the set of all possible boundary labelings $g$ of the $2n$ boundary vertices, and consider the vector space $R_n$ spanned by the abstract variables $r_g$, $g \in G$. We define an algebra structure on $R_n$ as follows: $r_{g_1} r_{g_2}$ is equal to 
\begin{enumerate}
    \item $r_g$, where $g$ is the boundary labeling obtained by combining the left half of $g_1$ and right half of $g_2$, if the right half of $g_1$ is obtained from the left part of $g_2$ via map $i \mapsto i'$;
    \item $0$ otherwise.
\end{enumerate} 

It is not hard to see that this product turns $R_n$ into an associative algebra with unity. Consider the map $\kappa: W_n \longrightarrow R_n$ defined by $\kappa: D \mapsto \sum_{g \in G_n} |L_{D,g}|_q r_g$.

\begin{theorem} \label{thm:rhom}
Map $\kappa$ is an injective algebra homomorphism, and so is map $\eta$ of Theorem \ref{thm:eta}.
\end{theorem}

\begin{proof}
The concatenation product  in $W_n$ is clearly compatible with the product structure of $R_n$. Thus in order to check that $\kappa$ is an algebra homomorphism, we need to verify that the defining relations of $W_n$ are satisfied in $R_n$. In particular it is enough to check that spider reduction rules are compatible with $\kappa$.

The first two reduction rules from Figure \ref{fig:sp6} are easy to verify. For example, a closed loop produces two singularities. If the loop is oriented for example clockwise, and labeled by $i$ and $i'$, then the two singularities contribute the factor $q^{\frac{4-2i}{4}}$ each. Thus as $i$ ranges through possible values of $1,2,3$, the total factor contributed is $q^{-1}+1+q = [3]_q$ just as it should be according to spider rules.

Let us therefore deal with the third rule. There is only one way to label the sides in the square configuration from the third rule. Namely, the boundary must contain a pair of labels $i$ and $i'$ and another pair of $j$ and $j'$. The only distinction comes from the relative position of those labels, the two possibilities shown on the Figure \ref{fig:sp15}.

\begin{figure}[h!]
\begin{center}
\input{sp15.pstex_t}
\end{center}
\caption{}\label{fig:sp15}
\end{figure}

One can see that in each case every consistent labeling of some web containing the square configuration is in bijection with a consistent labeling of exactly one of the two possible resolutions. Moreover, the statistic $\alpha(f)$ is preserved. For example in the upper case, the singularity with $i$, $k$ on the left cancels out with singularity with $k'$, $i'$ on the right; while the singularity with $k$, $j$ on the right cancels out with singularity with $j'$, $k'$ on the left. Other cases are similar.

Now we want to deduce the injectivity, i.e., that $W_n$ can be realized inside $R_n$. For each $(4,3,2,1)$-avoiding permutation $w \in S_n$ pick a reduced decomposition $\bar w = \prod s_{i_j}$, and consider the monomial $e_{\bar w} = \prod e_{i_j}^{(1)}$ in $TLM_n^3$. We use a triangularity argument to show that images $\kappa(\eta(e_{\bar w}))$ are linearly independent. 

Let $D_{\bar w}$ be the web obtained by concatenation of webs of generators $e_{i_j}^{(1)}$ according to $\bar w$. It is a well-known result (going back to Erd\"os) that every $(4,3,2,1)$-avoiding permutation can be partitioned into $3$ increasing subsequences. For a given $w$ pick one such partitioning and label the boundary of $D_{\bar w}$ according to this partitioning, obtaining $g_{w}$. For example, if $n=4$, $w = (1,4,3,2)$ and the partitioning is $(1,4) \cup (2) \cup (3)$, label the sources with $1,2,3,1$ and the sinks with $1',1',3',2'$ top to bottom. 

Note that the boundary labeling vector $r_{g_{w}}$ occurs in the decomposition of $\kappa(\eta(e_{\bar w}))$. To see this fact, label each diagram $D_{s_i}$ (constituting part of $D_{\bar w}$) so that the output labels are transposed input labels. Since $\bar w$ is a reduced decomposition, and since in $g_w$ entries having the same label are increasing, the resulting labeling is consistent. On the other hand, any permutation that produces $g_w$ can be written as combination of $w$ and some further transpositions between entries with same labels. The length of resulting permutation is bigger than that of $w$. Such a permutation cannot possibly be achieved by skipping some steps in ${\bar w}$. Therefore  $r_{g_{w}}$ occurs in decomposition of $\kappa(\eta(e_{\bar w}))$ with a non-zero coefficient (in fact with coefficient $1$).

Now take any extension of the Bruhat order. Then the boundary labeling vector $r_{g_{w}}$ cannot occur in decomposition of any $\kappa(\eta(e_{\bar v}))$ for $v<w$ in the chosen order. This essentially was proven above, given the sub-word characterization of the Bruhat order. Therefore the $\kappa(\eta(e_{\bar w}))$ are indeed linearly independent and the dimension of the image of $TLM_n^3$ in $R_n$ is equal to the number of $(4,3,2,1)$-avoiding permutations in $S_n$. As we already know this is exactly the dimension of $TLM_n^3$ and the injectivity of both $\kappa$ and $\eta$ follows.

\end{proof}

Let $D \in \mathfrak M_n$ be a reducible web, and let $e_D= \sum c_i e_{D_i}$ be the decomposition of the corresponding element of $TLM_n^3$ into irreducibles.  Let us consider the process of reduction of $D$ using the spider rules. This can be viewed as a binary tree with a branching whenever we apply the third rule. When descending towards one of the leaves of the tree, a coefficient is accumulated via the first two spider rules. The coefficient $c_i$ is equal to the sum $\sum_l c_{l,i}$ of the coefficients corresponding to leaves $l$ of the tree in which we end up with $D_i$. The Figure \ref{fig:sp8} illustrates a possible tree (a fragment of the whole web is shown).

\begin{figure}[h!]
\begin{center}
\input{sp8.pstex_t}
\end{center}
\caption{}\label{fig:sp8}
\end{figure}

Let us now start with a consistent labeling of $D$. From the proof of Theorem \ref{thm:rhom} we know that when each of the spider reduction rules is applied, we get a map from the consistent labelings of the original web to the consistent labellings of the resulting web. Furthermore, at each branching point the current labeling dictates into which of the two branches we go. Thus we can define the {\it {type}} of the original labeling $f$ as the irreducible web $D_i$ we end up with. Note that the type of a labeling a priori might depend on the choice of spider reduction steps. It seems likely that it is actually independent of the choices made, however it is not necessary for the further argument. From now on we assume that for every possible web one possible branching sequence of reduction steps is chosen.

Let $g$ be a boundary labeling for $D$. Let $L_{D,D_i,g}$ denote the set of consistent labelings of $D$ with boundary $g$ and of type $D_i$.

\begin{theorem} \label{thm:ci}
The following holds: 
$$
c_i = 
\begin{cases}
\frac{|L_{D,D_i,g}|_q}{|L_{D_i, g}|_q} & \text{if $|L_{D_i, g}|>0$;}\\
0 & \text{otherwise.}
\end{cases}
$$
\end{theorem}

\begin{proof}
Choose a particular sequence of reductions producing a binary tree as above. The spider reduction rules provide a surjection from the consistent labelings of $D$ with boundary $g$ and of type $D_i$ onto the consistent labelings of the $D_i$-leaves. The relative $q$-size of the fiber of each letter is exactly the coefficient that appears by applying the spider rules when we descend into that particular leaf. Therefore the set $L_{D,D_i,g}$ of consistent labelings of $D$ with boundary $g$ of type $D_i$ gets partitioned into the union of sets with $q$-size $c_{l,i} |L_{D_i, g}|_q$ as $l$ runs over type $D_i$ leaves and $c_{l,i}$ is the coefficient created when descending into leaf $l$. Since by definition $c_i = \sum_l c_{l,i}$, we conclude the needed statement.
\end{proof}

The rest of the paper proceeds with $q=1$.

\section{Web immanants and total positivity} \label{sec:imm}

For a function $f: S_n \longrightarrow \mathbb C$ and an $n \times n$ matrix $X$ an {\it {immanant}} $\Imm_f(X)$ is defined by $$\Imm_f(X) = \sum_{w \in S_n} f(w) x_{1,w(1)} \dotsc x_{n,w(n)}.$$ We define { {web immanants}} by analogy with the {\it {Temperley-Lieb immanants}} of Rhoades and Skandera \cite{RS1}.

For each irreducible web $D \in \mathfrak M_n$ and $w \in S_n$ let $f_D(w)$ be the coefficient of $e_D$ in the image $\theta_3(w)$. Then the immanants $$\Imm_D(X) = \Imm_{f_D}(X) = \sum_{w \in S_n} f_D(w) x_{1,w(1)} \dotsc x_{n,w(n)}$$ are called {\it {web immanants}}. 

Following \cite{RS2} let $z_{[i,j]}$ denote the sum of all elements of a parabolic subgroup of $S_n$ generated by $s_i, \ldots, s_{j-1}$. We will make use of the following lemma.

\begin{lemma} \label{lem:mon}
\begin{enumerate}
    \item $\theta_3(z_{[i,i+1]})  = \theta_3(s_i+1) = e_i^{(1)}$;
    \item $\theta_3(z_{[i,i+2]}) = 2 e_i^{(2)}$;
    \item $\theta_3(z_{[i,i+k]}) = 0$ for $k>2$.
\end{enumerate}
\end{lemma}

\begin{proof}
The first part is clear from definition. For the second part one checks that $(s_i+1)(s_{i+1}+1)(s_i+1)-(s_i+1)=z_{[i,i+2]}$. Finally, for the third part one can check that $z_{[i,i+3]} = 6 e_i^{(3)}=0$ and for any $k>3$, $z_{[i,i+3]}$ is a factor of $z_{[i,i+k]}$ in $\mathbb CS_n$.
\end{proof}

Now we are ready to consider the properties of web immanants.


Recall that a real matrix is {\it {totally nonnegative}} if the determinants of all its minors are nonnegative, see for example \cite{FZ} and references there. An immanant is totally nonnegative if, when applied to any totally nonnegative matrix, it produces a nonnegative number. For example, by definition the determinant is totally nonnegative. The following theorem is similar to \cite[Theorem 3.1]{RS1} and \cite[Proposition 32]{LP}.

\begin{theorem} \label{thm:nn}
Web immanants are totally nonnegative.
\end{theorem}

The proof resembles the proof of \cite[Proposition 2]{RS2}. In particular we will need the following lemma.

\begin{lemma} \cite[Lemma 2.5]{RS1}, \cite[Theorem 2.1]{Ste1} \label{lem:ste}
Given a totally nonnegative matrix $X$, it is possible to choose a set $Z$ of elements of $\mathbb CS_n$ of the form $z = \prod z_{[i_k, j_k]}$ and nonnegative numbers $c_z$, $z \in Z$ so that $$\sum_{w \in S_n} x_{1,w(1)} \dotsc x_{n,w(n)} w = \sum_{z \in Z} c_z z.$$
\end{lemma}

With this we are ready to prove the theorem.

\begin{proof}
Let $X$ be a totally nonnegative matrix and let $\sum c_z z$ be the expression as in Lemma \ref{lem:ste}. Then $$\Imm_D(X) = \sum c_z f_D(z).$$ Note however that by Lemma \ref{lem:mon} $\theta_3(z)$ is a monomial in the $e_i^{(j)}$. According to spider reduction rules each such monomial is a nonnegative combination of the $e_D$. Therefore $f_D(z) \geq 0$ for all $z \in Z$ and $\sum c_z f_D(z) \geq 0$.  
\end{proof}

\section{Complementary minors} \label{sec:min}

For two subsets $I, J \subset [n]$ of the same cardinality denote $\Delta_{I,J}(X)$ the minor of an $n \times n$ matrix $X$ with row set $I$ and column set $J$. A set of minors is called {\it {complimentary}} if each row and column index participates in exactly one of the minors. 

Let $(I_1, J_1)$, $(I_2, J_2)$ and $(I_3, J_3)$ be a triple of complementary minors. Define the boundary labeling $g$ by the following rule: $I_1, I_2, I_3$ prescribe which of the source vertices are adjacent to edge sides labeled by $1$s, $2$s and $3$s correspondingly, while $J_1, J_2, J_3$ prescribe which of the sink vertices are adjacent to edge sides labeled by $1'$s, $2'$s and $3'$s correspondingly. The following theorem is similar to \cite[Proposition 4.3]{RS1} and \cite[Theorem 7]{LP}.

\begin{theorem} \label{thm:tri}
We have $$\Delta_{I_1, J_1}(X) \Delta_{I_2, J_2}(X) \Delta_{I_3, J_3}(X) = \sum |L_{D_i,g}| \Imm_{D_i}(X),$$ where the sum is taken over all irreducible webs $D_i \in \mathfrak M_n$. Web immanants form a basis for the vector space generated by triples of complementary minors.
\end{theorem}

\begin{example}
The fact that $|L_{D,g}|=1$ in example on Figure \ref{fig:sp16} means that when the product of minors $$\left|  \begin{array}{cc}
x_{1,1} & x_{1,3}\\
x_{4,1} & x_{4,3}  
\end{array} \right|
\cdot x_{2,2} \cdot x_{3,4}$$ is decomposed into web immanants the coefficient of $\Imm_D$ for this particular irreducible web $D$ is equal to $1$.
\end{example}

Thus we have a positive combinatorial rule for expressing the products of triples of complementary minors in terms of web immanants. Note that unlike in Temperley-Lieb case the expression is not necessarily multiplicity-free, since it can happen that $|L_{D,g}|>1$. 

\begin{proof}
Let us start with a wiring diagram of a permutation $w$ corresponding to the reduced decomposition $\bar w = \prod s_{i_j}$. Then $\theta_3(w) = \prod (e_{i_j}^{(1)}-1)$ is an alternating sum $\sum c_D e_{D}$ of the $e_D$, where each web $D$ is obtained from the wiring diagram by uncrossing all crossings in one of the two ways, as shown on Figure \ref{fig:sp10}. We refer to them as {\it {vertical}} and {\it {horizontal}} uncrossings. Each horizontal uncrossing produces a minus sign coming from $-1$ in $e_{i_j}^{(1)}-1$. 

\begin{figure}[h!]
\begin{center}
\input{sp10.pstex_t}
\end{center}
\caption{}\label{fig:sp10}
\end{figure}

By Theorem \ref{thm:ci} we know that the coefficient in $e_D$ of $e_{D_i}$ for a particular irreducible web $D_i$ is equal to  $\frac{|L_{D,D_i,g}|}{|L_{D_i, g}|}$ for $D_i$-s such that $|L_{D_i,g}| \not = 0$, where $g$ can be chosen to be the boundary condition above. Then the coefficient of $x_{1,w(1)} \dotsc x_{n,w(n)}$ in $\sum |L_{D,g}| \Imm_D(X)$ is equal to $$\sum_{D, D_i, |L_{D_i,g}| \not = 0} c_D \frac{|L_{D,D_i,g}|}{|L_{D_i, g}|}  |L_{D_i,g}| = \sum_{D,D_i} c_D |L_{D,D_i,g}|.$$ Note that if $|L_{D_i,g}|=0$ then $|L_{D,D_i,g}|=0$ and thus sum in the formula can be taken over all irreducible $D_i$. Now, for a given $D$, $$\sum_{D_i} |L_{D,D_i,g}| = |L_{D,g}|.$$ Therefore, the coefficient of $x_{1,w(1)} \dotsc x_{n,w(n)}$ in $\sum |L_{D,g}| \Imm_D(X)$ is equal to the alternating sum $$\sum c_D |L_{D,g}|.$$

Note that there are two essentially different ways to label consistently a vertical uncrossing, as shown on Figure \ref{fig:sp11}. 

\begin{figure}[h!]
\begin{center}
\input{sp11.pstex_t}
\end{center}
\caption{}\label{fig:sp11}
\end{figure}

We refer to the first way as {\it {unstable}}, and to the second way as {\it {stable}}. Similarly, we refer to a horizontal uncrossing as {\it {stable}} if the labels on its two edges are equal, and {\it {unstable}} otherwise. Observe that every unstable uncrossing can be changed into a unique unstable uncrossing of the opposite kind, i.e., vertical to horizontal and horizontal to vertical.

Choose a planar embedding of the original wiring diagram of $w$ which does not have two crossings on the same vertical line. We define an involution on the set of all labelings of all possible uncrossed diagrams $D$ entering  $\sum c_D e_{D}$ as follows. Choose the leftmost {{unstable uncrossing}}. Swap the type of uncrossing, changing the labeling correspondingly. It is easy to see that this gives an involution.

Note that the two webs carrying the original and the resulting labelings enter $\sum c_D e_{D}$ with distinct signs, since one contains one more horizontal uncrossing than the other. Therefore corresponding terms in $\sum c_D |L_{D,g}|$ cancel out. The only terms that remain are the ones with all uncrossings stable. There is at most one such uncrossing/labeling, and it must have the following properties:
\begin{enumerate}
     \item if source $m$ is adjacent to label $i$ then sink $w(m)$ is adjacent to label $i'$ (here we say that $w$ {\it {agrees}} with $g$);
     \item all wires originating in sources with the same label uncross horizontally, all wires originating in sources with different labels uncross vertically.
\end{enumerate}
Then if the number of horizontal uncrossings is $l$, the resulting coefficient $c_D$ is given by
 $$
c_D = 
\begin{cases}
(-1)^l & \text{if $w$ agrees with $g$;}\\
0 & \text{otherwise.}
\end{cases}
$$
This number is exactly the coefficient of $x_{1,w(1)} \dotsc x_{n,w(n)}$ in $\Delta_{I_1, J_1}(X) \Delta_{I_2, J_2}(X) \Delta_{I_3, J_3}(X)$.

In order to see that web immanants form a basis it suffices to note that according to \cite{DKR} and the properties of Robinson-Schensted-Knuth insertion algorithm, the dimension of the space generated by products of triples of complementary minors is exactly the number of $(4,3,2,1)$-avoiding permutations.
\end{proof}

\section{Relation to Temperley-Lieb immanants} \label{sec:tl}

For a $(3,2,1)$-avoiding permutation $w$ and a permutation $v$ let $f_w(v)$ be the coefficient of $e_w$ in $\theta_2(v)$. In \cite{RS1} the {\it {Temperley-Lieb immanants}} were defined as $$\Imm_w^{TL}(X) = \sum_{v \in S_n} f_w(v) x_{1,v(1)} \dotsc x_{n,v(n)}.$$

Recall that each $(3,2,1)$-avoiding permutation $w$ corresponds to a non-crossing matching on $2n$ vertices, which is the Kauffman diagram for the basis element $e_w$ of the Temperley-Lieb algebra. By abuse of notation we denote this matching also as $w$. We turn it into an $A_1$-web by dropping an extra vertex on the edges which have both ends on the left or both ends on the right. A {\it {consistent labeling}} of an $A_1$ web is an assignment of labels $1,1',2,2'$ to sides of edges so that 

\begin{enumerate}
   \item every edge is labeled by $i,i'$;
   \item every internal vertex is adjacent either to $1,2$ or to $1',2'$.
\end{enumerate}

Let $(I_1,J_1)$ and $(I_2,J_2)$ be a pair of complementary minors, and let $g$ be the corresponding boundary labeling with $1,1',2,2'$. Let $M_{w,g}$ denote the set of consistent labelings of $w$ that are compatible with $g$. It is easy to see that $M_{w,g}$ is either empty or contains exactly one labeling. The following property of Temperley-Lieb immanants was proved in \cite[Proposition 4.3]{RS1}.

\begin{theorem} \label{thm:tl}
$$\Delta_{I_1,J_1}(X)\Delta_{I_2,J_2}(X) = \sum_w |M_{w,g}| \Imm_w^{TL}(X).$$  
\end{theorem}

Now let $(I,J)$  and $(I_3,J_3)$ be a pair of complementary minors. Let $\Imm_w^{TL}(X')$ be a Temperley-Lieb immanant of the submatrix $X'$ of $X$ with row set $I$ and column set $J$. Since $\Imm_w^{TL}(X')$ lies in the subspace of products of pairs of complementary minors of $X'$, the product $\Imm_w^{TL}(X') \Delta_{I_3,J_3}(X)$ lies in the subspace of products of triples of complementary minors of $X$. Therefore it must be expressible in terms of web immanants: $$\Imm_w^{TL}(X') \Delta_{I_3,J_3}(X) = \sum_D a_{w, I_3, J_3}^D \Imm_D(X).$$ 

There exists a {\it {forgetful map}} from consistent labelings of webs to consistent labelings of $A_1$-webs, given by deleting all edges labeled with $(3,3')$ and ignoring the loops, should any appear. Let us denote by $L_{D,g,w}$ the set of consistent labelings of a web $D$ compatible with the boundary labeling $g$ and mapped by the forgetful map to a consistent labeling of $w$. 

Let $\tilde g$ be a boundary labeling with positions of $3$s and $3'$s given  by $(I_3, J_3)$ and such that $M_{w, g}$ is non-empty. The following theorem gives an interpretation of the transition coefficients $a_{w, I_3, J_3}^D$.

\begin{theorem}
The size of $L_{D,\tilde g,w}$ does not depend on the particular choice of $\tilde g$ and we have $a_{w, I_3, J_3}^D= |L_{D,\tilde g,w}|$.
\end{theorem} 

\begin{example}
For the irreducible web on Figure \ref{fig:sp16} the shown labeling is the only one having $I_3=\{3\}$, $J_3 = \{4\}$ and mapped by the forgetful map to the $A_1$-web corresponding to $w = (2,3,1)$, cf. Figure \ref{fig:sp17}. 

\begin{figure}[h!]
\begin{center}
\input{sp17.pstex_t}
\end{center}
\caption{}\label{fig:sp17}
\end{figure}

Thus the coefficient of $\Imm_D(X)$ in $\Imm_w^{TL}(X') \cdot x_{3,4}$ is $1$.
\end{example}

\begin{proof}
Note that any two consistent labelings of $w$ can be obtained one from the other by changing all labels along several of the edges of $w$. This gives a bijection between consistent labelings of $w$ for different boundary conditions $g$, with the only requirement that $M_{w,g}$ is non-empty. This bijection can be lifted to elements of the $L_{D,\tilde g,w}$ by simply doing the same changes on corresponding edges. This shows independence of the choice of $\tilde g$.

Now define alternative immanants $$\Imm_w^a (X) = \sum_D |L_{D,\tilde g,w}| \Imm_D(X).$$ Let $(I_1,J_1)$ and $(I_2,J_2)$ be a pair of complementary minors of $X'$ and let $g$ be the corresponding boundary labeling. We know from Theorem \ref{thm:tri} that $$\Delta_{I_1, J_1}(X) \Delta_{I_2, J_2}(X) \Delta_{I_3, J_3}(X) = \sum_D |L_{D,g}| \Imm_D(X).$$ However, $|L_{D,g}| = \sum_w |M_{w,g}| |L_{D,g,w}|$. Then we get $$\sum_D |L_{D,g}| \Imm_D(X) =  \sum_D \sum_w |M_{w,g}| |L_{D,g,w}| \Imm_D(X).$$ However, $|M_{w,g}| |L_{D,g,w}| = |M_{w,g}| |L_{D, \tilde g,w}|$ since if $M_{w,g}$ is non-empty we argued above that $|L_{D, g,w}| = |L_{D, \tilde g,w}|$ and otherwise both sides are $0$. Therefore $$\sum_D \sum_w |M_{w,g}| |L_{D,g,w}| \Imm_D(X) = \sum_w |M_{w,g}| \sum_D  |L_{D, \tilde g,w}| \Imm_D(X) = \sum_w |M_{w,g}| \Imm_w^a (X).$$On the other hand, from Theorem \ref{thm:tl} we know that $$\Delta_{I_1, J_1}(X) \Delta_{I_2, J_2}(X) \Delta_{I_3, J_3}(X) = \sum_w |M_{w,g}| \Imm_w^{TL}(X') \Delta_{I_3, J_3}(X).$$ From that we conclude by inverting that $\Imm_w^a (X) = \Imm_w^{TL}(X') \Delta_{I_3, J_3}(X)$ and thus $a_{w, I_3, J_3}^D= |L_{D,\tilde g,w}|$.

\end{proof}

\section{Weighted networks} \label{sec:net}

Let $G = (V,E)$ be a finite oriented acyclic planar graph with $n$ sources followed by $n$ sinks on the boundary of a Jordan curve. Let $\omega: E \longrightarrow R$ be a weight function assigning to each edge $e \in E$ the weight $\omega(e)$ in some commutative ring $R$. We refer to $N = (G, \omega)$ as to a {\it {weighted network}}. A {\it {path}} $p$ in $N$ is a path from a source to a sink, and $\omega(p) = \prod_{e \in p} \omega(e)$. Let $P(N)$ be the set of all paths in $N$.

Let ${\bf p} = (p_1, \ldots, p_{n})$ be a family of paths in $P(N)$ such that no four paths in ${\bf {p}}$ intersect in the same vertex. We denote $\omega({\bf p}) = \prod \omega(p_i)$. Removing all edges in $N$ which do not lie in any $p_i$, and marking as double or triple the edges used twice or thrice by $\bf p$, we get an underlying {\it {marked subnetwork}} $\tilde N({\bf p})$ of $N$. We denote by $\tilde N < N$ the fact that $\tilde N$ is a marked subnetwork of $N$, and we denote by $P(\tilde N)$ the set of all $\bf p$ such that $\tilde N = \tilde N({\bf p})$. 

Define a {\it {vertical uncrossing}} of a crossing of two or three paths by a procedure shown on Figure \ref{fig:sp12}. 

\begin{figure}[h!]
\begin{center}
\input{sp12.pstex_t}
\end{center}
\caption{}\label{fig:sp12}
\end{figure}

Define $D(\tilde N)$ to be the graph obtained by vertically uncrossing all the crossings in $\tilde N$. Then it is clear that $D(\tilde N) \in \mathfrak M_n$ is a (possibly reducible) web. Let $e_{D(\tilde N)} = \sum c_{i, \tilde N} e_{D_i}$ be the decomposition into irreducibles. 

Let ${\bf I} = (I_1, I_2, I_3)$ and ${\bf J} = (J_1, J_2, J_3)$ be disjoint partitions of $[n]$ such that $|I_k|=|J_k|$. Let $P_{\bf I,J}(N)$ be the set of families $\bf p$ such that paths which start in $I_k$ end in $J_k$, and the paths which start in the same $I_k$ do not intersect. 

Let $X(N)$ be the matrix given by $x_{i,j} = \sum \omega(p)$, where the sum is taken over all $p$ starting at $i$th source and ending at $j$th sink. The following statement is known as {\it {the Lindstr\"om lemma}}, cf. \cite{FZ}.

\begin{lemma}
The determinant $\Delta(X(N))$ is equal to $\sum_{\bf p} \omega({\bf p})$, where the sum is taken over all pairwise non-intersecting families of paths $\bf p$ in $P(N)$.
\end{lemma}

Let $$\Imm'_{D_i}(N) = \sum_{\tilde N < N} c_{i, \tilde N} \omega(\tilde N).$$  Let $g$ be the boundary labeling determined by $(\bf I, J)$ as before. The following theorem is similar to \cite[Proposition 26]{LP}.

\begin{theorem}
We have $$\Delta_{I_1, J_1}(X(N)) \Delta_{I_2, J_2}(X(N)) \Delta_{I_3, J_3}(X(N)) = \sum_i |L_{D_i,g}| \Imm'_{D_i}(N).$$
\end{theorem}

\begin{proof}
It is clear from the Lindstr\"om lemma that $$\Delta_{I_1, J_1}(X(N)) \Delta_{I_2, J_2}(X(N)) \Delta_{I_3, J_3}(X(N)) = \sum_{{\bf p} \in P_{\bf I, J}(N)} \omega({\bf p}).$$ Note that the sum on the right involves only families $\bf p$ with no four paths crossing in one point. Label each path with $k$-s and $k'$-s if it starts at $I_k$. Then the induced labeling of $D(\tilde N({\bf p}))$ is consistent labeling, and furthermore this map is a bijection between $\bigcup_{\tilde N < N} L_{D(\tilde N), g}$ and $P_{\bf I,J}(N)$. Thus $$\sum_{{\bf p} \in P_{\bf I, J}(N)} \omega({\bf p}) = \sum_{\tilde N < N} |L_{D(\tilde N), g}| \omega(\tilde N).$$ Recall from the proof of Theorem \ref{thm:ci} that $$|L_{D(\tilde N), g}| = \sum_{D_i} |L_{D(\tilde N), D_i, g}| = \sum_{D_i} c_{i, \tilde N} |L_{D_i, g}|.$$ Then $$\sum_{\tilde N < N} |L_{D(\tilde N), g}| \omega(\tilde N) = \sum_{\tilde N < N} (\sum_{D_i} c_{i, \tilde N} |L_{D_i, g}| \omega(\tilde N))$$ $$=\sum_{D_i} (|L_{D_i, g}| \sum_{\tilde N < N} c_{i, \tilde N} \omega(\tilde N)) = \sum_{D_i} |L_{D_i,g}| \Imm'_{D_i}(N).$$
\end{proof}

\begin{corollary}
We have $\Imm_D(X(N)) = \Imm'_D(N)$.
\end{corollary}

\begin{proof}
The products of the complementary minors labeled by the {\it {standard bitableaux}} of \cite{DKR} with at most three columns form a linear basis, the {\it {standard basis}},  for the subspace of immanants generated by products of triples of complementary minors. On the other hand we know that the number of those is exactly the dimension of $TLM_n^3$, i.e., the number of irreducible webs in $\mathfrak M_n$. Therefore the transition matrix from the $\Imm_D$ to the standard basis is invertible. Then both the $\Imm_D$ and the $\Imm'_D$ are recovered via the same transition matrix from the standard basis, and thus they must coincide.
\end{proof}

Note that this provides an alternative proof of Theorem \ref{thm:nn} since by a result of Brenti \cite{Br} every totally nonnegative matrix can be represented by a planar weighted network with nonnegative weights. In fact we have implicitly used the result of Brenti in the original proof of Theorem \ref{thm:nn} as well, when we relied on Lemma \ref{lem:ste}.

\section{Concluding remarks} \label{sec:con}

In \cite{RS2} Rhoades and Skandera introduced a family of immanants called {\it {Kazhdan-Lusztig immanants}}, where the coefficients of monomials are given by evaluations of Kazhdan-Lusztig polynomials. Kazhdan-Lusztig immanants are labeled by permutations, and constitute a basis for the whole space of immanants. In \cite{RS2} it is shown, relying on the work of Fan and Green \cite{FG}, that Temperley-Lieb immanants coincide with the Kazhdan-Lusztig immanants for $(3,2,1)$-avoiding permutations. According to \cite[Theorem 2.4]{RS3} the Kazhdan-Lusztig immanants labeled by $(k, \ldots, 1)$-avoiding permutations constitute a basis for the vector space generated by products of $k$-tuples of complementary minors. Thus one is naturally led to wonder what is the relation between $A_2$-web immanants and $(4,3,2,1)$-avoiding Kazhdan-Lusztig immanants. This question might be related to the question addressed in \cite{KKh}.

A theme of Schur positivity appears in the study of immanants, cf. \cite{Ste2, H, RS2, LPP}. In the terminology of \cite{H, RS2} a {\it {generalized Jacobi-Trudi matrix}} corresponding to two partitions $\lambda, \mu$ is the matrix with entries $x_{i,j} = h_{\lambda_i-\mu_j}$, where the $h$ are the {\it {complete homogeneous symmetric functions}}, cf. \cite{St}. It was shown in \cite{RS2}, relying on a result of Haiman \cite{H}, that Kazhdan-Lusztig immanants of generalized Jacobi-Trudi matrices are nonnegative when expressed in the basis of {\it {Schur functions}}. One might wonder if web immanants have the same property. Note that if web immanants were shown to be nonnegative combinations of Kazhdan-Lusztig immanants, the Schur positivity would follow. 

It is natural to expect a generalization of the present results from $TLM_n^3$ to $TLM_n^k$ for any $k$. That would involve having a Kauffman diagram-like calculus for any $k$. One can however use the relations described in \cite{BK} to reverse engineer such a calculus. For that one would need to make a guess of what 
$A_k$-{\it {webs}} represent generators $e_j^{(i)}$ of $TLM_n^k$. The choice on Figure \ref{fig:sp4} seems to be natural, here there are $k-i-1$ edges in the middle. 

\begin{figure}[h!]
\begin{center}
\input{sp4.pstex_t}
\end{center}
\caption{}\label{fig:sp4}
\end{figure}

It is convenient to let such diagrams represent the rescaled generators $[i!]_q e_{j}^{(i)}$, where $[i!]_q = [1]_q [2]_q \dotsc [i]_q$. Given that, one can work out for example the conjecture for $TLM_n^4$ web rules to be as shown on Figure \ref{fig:sp5}. It seems that the crucial part of an argument would be verifying that the number of irreducible $A_k$-webs defined in this way is equal to the number of $(k, \ldots, 1)$-avoiding permutations. Once that is known, one could expect an argument similar to that of Theorem \ref{thm:rhom} to exist, implying linear independence of irreducible webs. 

\begin{figure}[h!]
\begin{center}
\input{sp5.pstex_t}
\end{center}
\caption{}\label{fig:sp5}
\end{figure}

\end{document}